\documentclass[12pt]{extarticle}
\usepackage{amsthm}
\usepackage{amsmath}
\usepackage{amsfonts}
\usepackage{graphicx}
\usepackage{hyperref}
\newtheorem{theorem}{Theorem}
\newtheorem{corollary}[theorem]{Corollary}
\newtheorem{lemma}[theorem]{Lemma}

\newcounter{claims}
\newenvironment{claims}{\refstepcounter{claims}\par\medskip\noindent%
{{\bf (\theclaims)}~~}}{\par\medskip}
\newcommand{\claim}[2]{\begin{claims}{\em #2}\label{#1}\end{claims}}
\newcommand{\refclaim}[1]{(\ref{#1})}
\newcommand{\GG}{{\mathcal G}}
\newcommand{\Ss}{{\mathcal S}}

\begin{document}

\title{$3$-coloring triangle-free planar graphs with a precolored $8$-cycle}
\author{
Zden\v{e}k Dvo\v{r}\'{a}k \thanks{Charles University in Prague, E-mail: {\tt rakdver@iuuk.mff.cuni.cz}. Supported by
the Center of Excellence -- Inst. for Theor. Comp. Sci., Prague (project P202/12/G061 of Czech Science Foundation), and
by project LH12095 (New combinatorial algorithms - decompositions, parameterization, efficient solutions) of Czech Ministry of Education.}
\and
Bernard Lidick\'{y} \thanks{University of Illinois at Urbana-Champaign, E-mail: {\tt lidicky@illinois.edu}}
}
\date{\today}

\maketitle
\begin{abstract}
Let $G$ be a planar triangle-free graph and let $C$ be a cycle in $G$ of length at most $8$.
We characterize all situations where a $3$-coloring of $C$ does not extend to a
proper $3$-coloring of the whole graph.
\end{abstract}

Graphs in this paper may have loops or parallel edges.
A (proper) \emph{$k$-coloring} of a graph $G=(V,E)$ is a mapping $\varphi:V \rightarrow \{1,\ldots,k\}$
such that $\varphi(u)\neq\varphi(v)$ whenever $uv$ is an edge of $G$.
A graph $G$ is \emph{$k$-colorable} if there exists a $k$-coloring of $G$.

Deciding whether a graph is $k$-colorable is NP-complete~\cite{garey1979computers} for every $k\ge 3$.
The situation is somewhat different for planar graphs, which are $4$-colorable by the well-known
Four Color Theorem~\cite{AppHak1,AppHakKoc,rsst}.  However, $3$-colorability of planar graphs is NP-complete~\cite{bib-dai},
which motivates study of additional assumptions guaranteeing $3$-colorability.  For instance, Gr\"{o}tzsch theorem~\cite{grotzsch1959}
states that every triangle-free planar graph is $3$-colorable, inspiring many related results.

Gimbel and Thomassen~\cite{gimbel} proved that a triangle-free projective-planar graph is $3$-colorable unless it
contains a non-bipartite quadrangulation.  We say that a graph $G$ is \emph{$k$-critical} if it is not $(k-1)$-colorable but
every proper subgraph of $G$ is $(k-1)$-colorable; thus, the previous result could be restated as the claim that every $4$-critical
triangle-free projective-planar graph is a non-bipartite quadrangulation.  Critical graphs and study of their properties give
important tools for both theory and algorithms for graph coloring.  For example, many coloring algorithms (especially for embedded
graphs) are based on detection of particular critical subgraphs.

Let us give a quick overview of results regarding embedded critical graphs.  While there are infinitely many $5$-critical graphs
embeddable in any fixed surface except for the sphere, Thomassen~\cite{Thomassen97} proved that for every $g\ge 0$ and $k\ge 6$, there are only finitely
many $k$-critical graphs of Euler genus $g$.  This result was later improved by Postle and Thomas~\cite{pothom} by showing that such $6$-critical
graphs have size $O(g)$.  For graphs of girth at least five, there are only finitely many $k$-critical graphs of Euler genus $g$ for every $g\ge 0$ and $k\ge 4$,
as proved by Thomassen~\cite{thomassen-surf}; the bound on the size of such graphs was later improved to $O(g)$ by Dvo\v{r}\'ak et al.~\cite{trfree3}.
For triangle-free graphs, there are only finitely many $k$-critical graphs of Euler genus $g$ for every $g\ge 0$ and $k\ge 5$.
There are infinitely many triangle-free $4$-critical graphs of genus $g$ for every $g\ge 1$, however their structure is restricted as shown
by Dvo\v{r}\'ak et al.~\cite{trfree4}; in particular, one can design a linear-time algorithm to decide $3$-colorability of triangle-free graphs
of bounded genus~\cite{coltrfree}.

A problem that commonly arises in study of critical graphs is as follows. Suppose that $F$ is a $k$-critical graph and
$C$ is a (usually small) subgraph of $F$ whose removal disconnects $F$.  What can one say about the arising components?
This motivates the following definition.
Let $G$ be a graph and $C$ its (not necessarily induced) proper subgraph.
We say that $G$ is \emph{$C$-critical for $k$-coloring} if for every proper subgraph $H\subset G$
such that $C\subseteq H$, there exists a $k$-coloring of $C$ that extends
to a $k$-coloring of $H$, but not to a $k$-coloring of $G$.

Notice that $(k+1)$-critical graphs are exactly $C$-critical graphs for $k$-coloring with $C=\emptyset$.
Furthermore, it is easy to see that if $F$ is a $(k+1)$-critical graph and $F=G\cup G'$, where $C=G\cap G'$,
then either $G=C$ or $G$ is $C$-critical for $k$-coloring.  A variation on this claim that is often useful
when dealing with embedded graphs is as follows.

\begin{lemma}[Dvo\v{r}\'ak et al.~\cite{trfree1}]\label{lemma-critin}
Let $G$ be a plane graph with outer face $K$.  Let $C$ be a cycle in $G$ that does
not bound a face, and let $H$ be the subgraph of $G$ drawn in the closed disk bounded by $C$.
If $G$ is $K$-critical for $k$-coloring, then $H$ is $C$-critical for $k$-coloring.
\end{lemma}

Another way to view $C$-critical graphs is the precoloring extension perspective.
Suppose that we are given a coloring $\psi$ of a proper subgraph $C$ of $G$, does there exist
a coloring of $G$ that matches $\psi$ on $C$?  To answer this question, it suffices to consider
the colorings of a maximal $C$-critical subgraph of $G$.

As suggested by Lemma~\ref{lemma-critin}, an important case of the precoloring extension problem
is the one where all precolored vertices are incident with one face of a plane graph (without loss of generality
the outer one).  From now on, we only deal with $3$-colorings in this paper, and thus we usually omit the qualifier
``for $3$-coloring'' when speaking about $C$-critical graphs.

Plane critical graphs of girth 5 with a precolored face of length at most $11$ were enumerated
by Walls~\cite{walls} and independently by Thomassen~\cite{thomassen-surf},
who also gives some necessary conditions for plane graphs of girth 5 with a precolored face of length $12$.  
The exact enumeration of plane graphs of girth 5 with a precolored face of length $12$ appears in
Dvo\v{r}\'ak and Kawarabayashi~\cite{dk}. 
The number of critical graphs grows exponentially with the length of the
precolored face, and enumerating all the plane critical graphs becomes increasingly difficult.
Dvo\v{r}\'ak and Lidick\'{y}~\cite{dvolid} implemented an algorithm to generate such 
plane graphs of girth 5 based on the results of
Dvo\v{r}\'ak and Kawarabayashi~\cite{dk}, and used the computer to enumerate
the plane critical graphs of girth 5 with the outer face of length at most $16$.

In this paper, we consider the same question in the setting of triangle-free graphs.
Aksenov~\cite{aksenov} showed that any precoloring of a face of length
at most $5$ in a plane triangle-free graph extends to the whole graph. Gimbel  and Thomassen~\cite{gimbel} characterized 
plane critical graphs of girth 4 with a precolored face of length $6$.
The faces of a plane graph distinct from the outer one are called \emph{internal}.  
\begin{theorem}[Gimbel and Thomassen~\cite{gimbel}]\label{thm-gimbel}
Let $G$ be a plane triangle-free graph with outer face bounded by a cycle $C=c_1c_2\ldots$ of length at most $6$.
The graph $G$ is $C$-critical if and only if $C$ is a $6$-cycle, all internal faces of $G$ have length exactly four
and $G$ contains no separating $4$-cycles.  Furthermore, if $\varphi$ is a $3$-coloring of $G[V(C)]$
that does not extend to a $3$-coloring of $G$, then $\varphi(c_1)=\varphi(c_4)$, $\varphi(c_2)=\varphi(c_5)$ and $\varphi(c_3)=\varphi(c_6)$.
\end{theorem}
The previous result was independently obtained by Aksenov, Borodin, and Glebov~\cite{aksenov6cyc}.
The characterization was further extended by Aksenov, Borodin, and Glebov~\cite{bor7c}
to a precolored face of length $7$, see Corollary~\ref{corollary-7cyc} below.
In this paper, we give a simpler proof of this result by exploiting properties of nowhere-zero flows.

Furthermore, as the main result, we extend the characterization to the case of a precolored face of length $8$.
For a plane graph $G$, let $S(G)$ denote the multiset of lengths of the internal $(\ge\!5)$-faces of $G$.

\begin{theorem}\label{thm:8cycleprecise}
Let $G$ be a connected plane triangle-free graph with outer face bounded by a cycle $C$ of length $8$.  The graph $G$ is $C$-critical
if and only if $G$ contains no separating cycles of length at most five, the interior of every non-facial 6-cycle
contains only faces of length four and one of the following propositions is satisfied (see Figure~\ref{fig-8cycle} for an illustration).
\begin{itemize}
 \item[\textrm{(a)}] $S(G)=\emptyset$, or
 \item[\textrm{(b)}] $S(G)=\{6\}$ and the $6$-face of $G$ intersects $C$ in a path of length at least one, or
 \item[\textrm{(c)}] $S(G)=\{5,5\}$ and each of the $5$-faces of $G$ intersects $C$ in a path of length at least two, or
 \item[\textrm{(d)}] $S(G)=\{5,5\}$ and the vertices of $C$ and the $5$-faces $f_1$ and $f_2$ of $G$ can be labelled
 in clockwise order along their boundaries so that $C=c_1c_2\ldots c_8$, $f_1=c_1v_1zv_2v_3$ and $f_2=zw_1c_5w_2w_3$
 (where $w_1$ can be equal to $v_1$, $v_1$ can be equal to $c_2$, etc.)
\end{itemize}
\end{theorem}

\begin{figure}
\center{\includegraphics[scale=1]{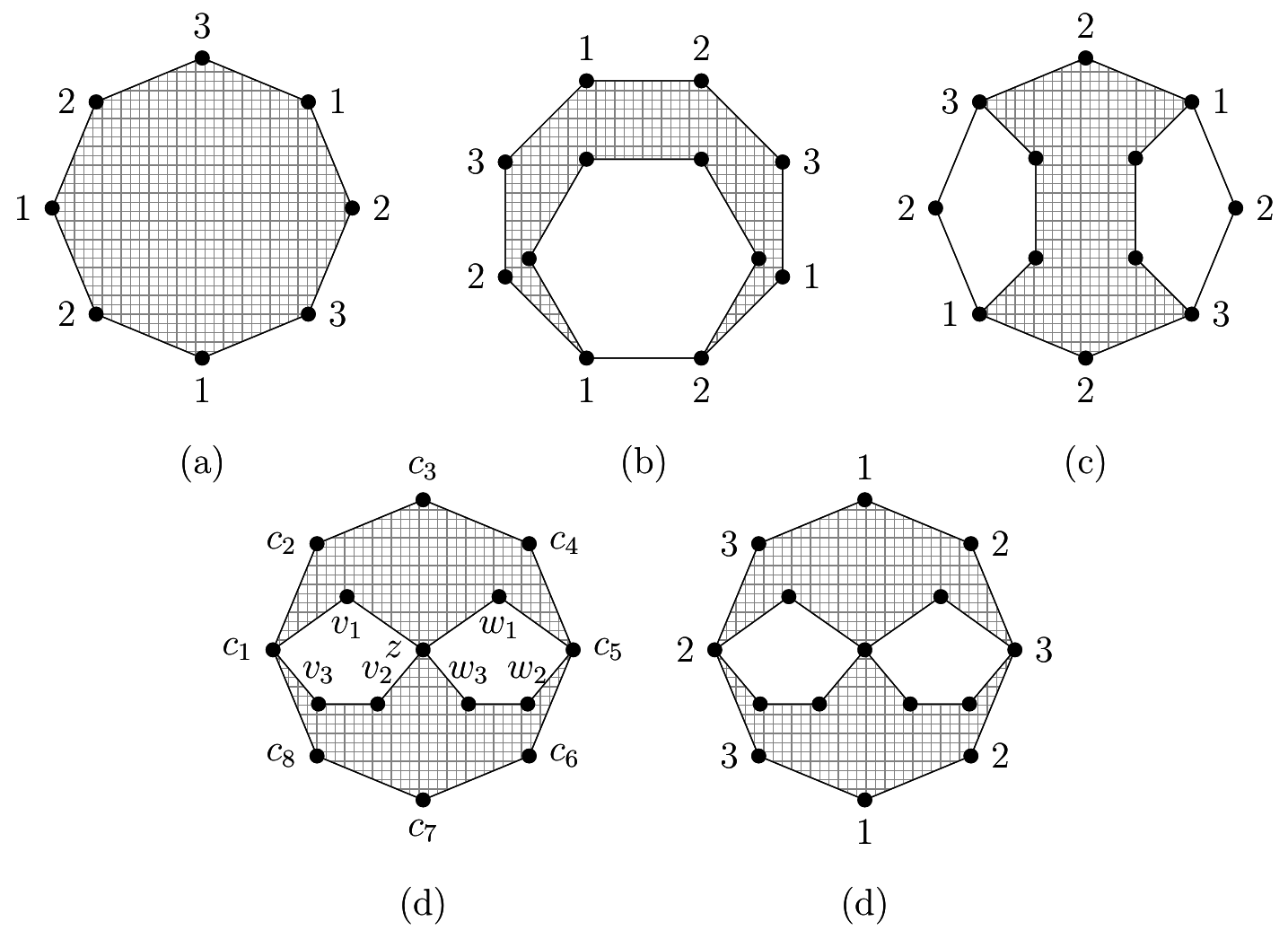}}
\caption{Graph described by Theorem~\ref{thm:8cycleprecise} and examples of $3$-colorings of $C$ that do not extend.}\label{fig-8cycle}
\end{figure}

We describe a connection between nowhere-zero flows and 3-colorings in Section~\ref{sec:flows}. We follow with Section~\ref{sec:faces}
constraining lengths of faces in critical graphs.
Section~\ref{sec:k-2} is devoted to the special case where the precolored
face has length $k$, one other face has length $k-2$ and the rest are 4-faces.
Finally, in Section~\ref{sec:8face} we give a proof of Theorem~\ref{thm:8cycleprecise}.

\section{3-colorings and nowhere-zero flows}\label{sec:flows}

Let $H$ be a connected plane graph without loops and let $H^\star$ be the dual of $H$.
Let $\varphi$ be a $3$-coloring $\varphi$ of $H$ by colors $\{1,2,3\}$.
Let us define an orientation of $H^\star$ as follows.  Let $e\in E(H)$ be an edge
incident with a vertex $u\in V(H)$ and let $f$ and $h$ be the faces incident with $e$,
appearing as $f,e,h$ in the clockwise order around $u$ in the drawing of $H$.
Let $v$ be the other vertex of $G$ incident with $e$ and let $e^\star$ be the edge of $H^\star$ corresponding to $e$.
The edge $e^\star$ is oriented towards $h$ if and only if $\varphi(u)-\varphi(v)\in \{1,-2\}$.
Suppose that $h$ is the outer face of $H$ and $h$ is bounded by a cycle; if $e^\star$ is oriented towards $h$, then we say that $e$ is
a \emph{sink edge}, otherwise $e$ is a \emph{source edge}.  Note that whether an edge is source
or sink depends only on the restriction of $\varphi$ to the boundary of the outer face.

Since $\varphi$ is a proper coloring, every edge of $H^\star$ has an orientation.
As shown by Tutte~\cite{tutteflow}, this orientation of $H^\star$ defines a nowhere-zero $\mathbb Z_3$-flow
that is, for each $f\in V(H^\star)$, the in-degree and the out-degree of $f$ differ by a multiple of $3$;
and conversely, each nowhere-zero $\mathbb Z_3$-flow in $H^\star$ defines a $3$-coloring of $H$,
uniquely up to a rotation of colors.  See e.g.~\cite{diestel} for more details.

Consider a face $f$ of $H$, and let $\delta(f)$ be the difference between its in-degree and out-degree when
considered as a vertex of $H^\star$.  Clearly, $\delta(f)$ and $|f|$ have the same parity and $|\delta(f)|\le |f|$.
Since $\delta(f)$ is a multiple of $3$, if $f$ is a $4$-face, then $\delta(f)=0$.
Similarly, if $|f|=5$ or $|f|=7$, then $|\delta(f)|=3$, and if $|f|=6$ or $|f|=8$, then
$|\delta(f)|\in \{0,6\}$.

Let $G$ be a connected plane triangle-free graph with the outer face $C$ bounded by a cycle.
We say that a function $q$ assigning an integer to each internal face of $G$ is a \emph{layout}
if each internal face $f$ satisfies $|q(f)|\le|f|$, $q(f)$ is divisible by $3$ and has the same parity as $|f|$.
In particular, $q(f)=0$ for every $4$-face, and thus it suffices to specify the values of $q$ for
$(\ge\!5)$-faces of $G$.  Consider a proper $3$-coloring $\psi$ of $C$, let $n_s$ be the number of
source edges and $n_t$ the number of sink edges of $C$ with respect to $\psi$ and let $m$ be the sum of the values of $q$
over all internal faces of $G$.  We say that $q$ is \emph{$\psi$-balanced} if $n_s+m=n_t$.
We define a graph $G^{q,\psi}$ as follows.  The vertex set of $G^{q,\psi}$ consists of the internal faces of $G$
and of two new vertices $s$ and $t$.  The adjacencies between vertices of  $V(G^{q,\psi})\setminus \{s,t\}$
are the same as in the dual of $G$.  For each internal face $f$ with $q(f)>0$, $s$ is joined to $f$ by $q(f)$
parallel edges.  For each internal face $f$ with $q(f)<0$, $t$ is joined to $f$ by $-q(f)$
parallel edges.  For each edge $e\in E(C)$ incident with an internal face $f$,
$f$ is joined to $t$ if $e$ is a sink edge with respect to $\psi$, and $f$ is joined to $s$ otherwise.
See Figure~\ref{fig-network} for an illustration.
We say that $s$ and $t$ are the \emph{terminals} of $G^{q,\psi}$ and we write $c(q,\psi)$ for the degree of $s$ in $G^{q,\psi}$.
Note that $q$ is $\psi$-balanced if and only if $s$ and $t$ have the same degree.

\begin{figure}
\center{\includegraphics[scale=1]{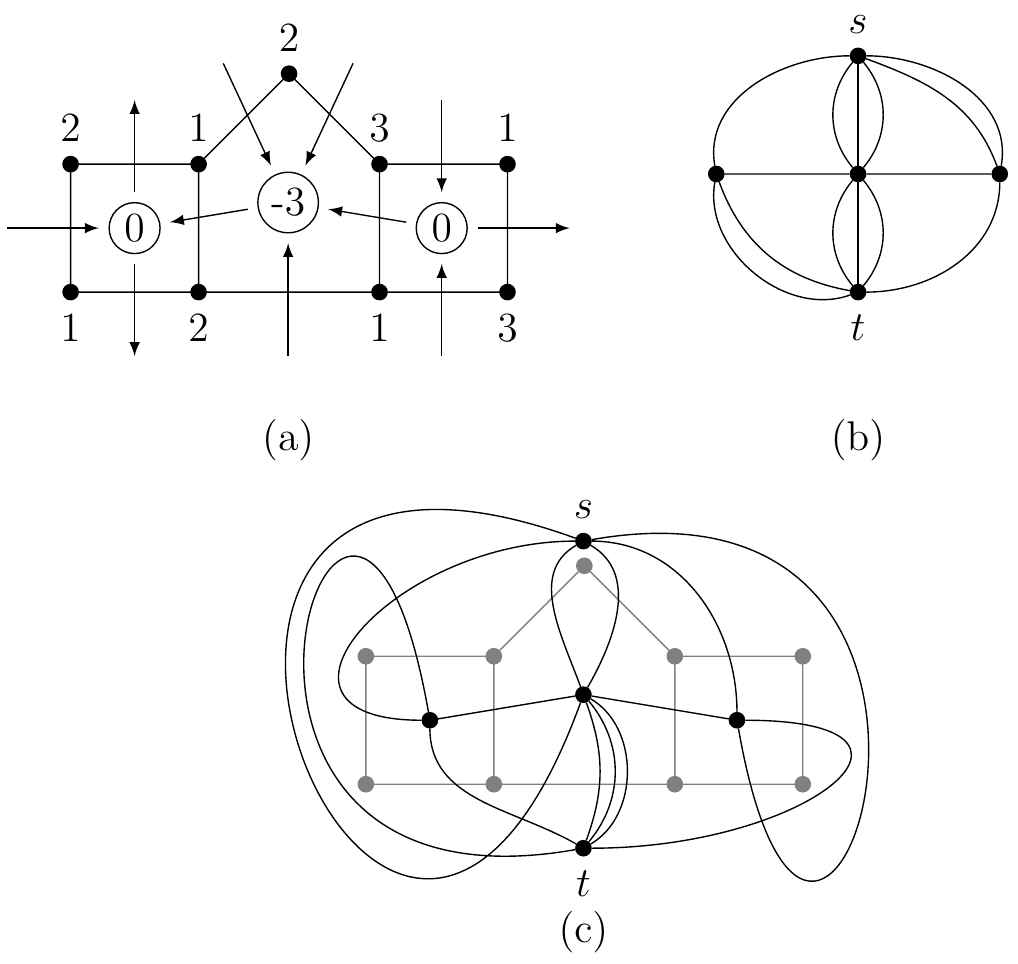}}
\caption{(a): A $3$-coloring $\psi$ of a graph $G$ with the corresponding orientation of edges of $G^\star$ and the balanced layout $q$. (b): the corresponding graph $G^{q,\psi}$. (c)  $G^{q,\psi}$ together with $G$.}\label{fig-network}
\end{figure}

\begin{lemma}\label{lemma-flowcol}
Let $G$ be a connected plane triangle-free graph with the outer face $C$ bounded by a cycle and let $\psi$ be
a $3$-coloring of $C$.  The coloring $\psi$ extends to a $3$-coloring of $G$ if and only if
there exists a $\psi$-balanced layout $q$ such that the terminals of $G^{q,\psi}$
are not separated by an edge-cut smaller than $c(q,\psi)$.
\end{lemma}
\begin{proof}
Suppose first that $\psi$ extends to a $3$-coloring $\varphi$ of $G$.  Let us orient
$G^\star$ according to $\varphi$ as described at the beginning of the section.  For each internal face $f$ of $G$, let $q(f)$ be the
difference between the in-degree and the out-degree of $f$.  Clearly, $q$ is a layout.
Let $s$ and $t$ be the terminals of $G^{q,\psi}$, let us orient the edges of $G^{q,\psi}$ incident
with $s$ away from $s$ and the edges incident with $t$ towards $t$, and orient all the other
edges of $G^{q,\psi}$ in the same way as in the orientation of $G^\star$.  Note that for each vertex
$f\in V(G^{q,\psi})\setminus\{s,t\}$, the in-degree of $f$ in $G^{q,\psi}$ is equal to its out-degree.
Consequently, $s$ and $t$ have the same degree and the orientation of edges of $G^{q,\psi}$ defines a
flow of size $c(q,\psi)$ from $s$ to $t$.  Therefore, every edge-cut between $s$ and $t$ in $G^{q,\psi}$
has capacity at least $c(q,\psi)$.

Let us now conversely assume that there exists a $\psi$-balanced layout $q$ such that the terminals of $G^{q,\psi}$
are not separated by an edge-cut smaller than $c(q,\psi)$.  By Menger's theorem,
$G^{q,\psi}$ contains $c(q,\psi)$ pairwise-edge disjoint paths $P_1$, \ldots, $P_{c(q,\psi)}$ from $s$ to $t$.
Let $G'=G^{q,\psi}-E(P_1\cup\ldots\cup P_{c(q,\psi)})$.  Since $q(f)$ has the same parity as $|f|$ for every internal
face $f$ of $G$, each vertex of $G'$ has even degree, and thus $G'$ is a union of pairwise edge-disjoint cycles.  For each such cycle, orient all its edges in one (arbitrary) direction.
For each path $P_i$ ($1\le i\le c(q,\psi)$), orient its edges towards $t$.  This defines an orientation of $G^{q,\psi}$,
which gives an orientation of $G^\star$ corresponding to a nowhere-zero $\mathbb Z_3$-flow consistent with the coloring $\psi$.
By the correspondence between flows and colorings, this defines a $3$-coloring of $G$ that extends $\psi$.
\end{proof}

Let us remark that for plane triangle-free graphs $G$ such that 
$$\sum_{\mbox{$f$ internal face of $G$}} (|f|-4)$$
is bounded by a constant $c$, Lemma~\ref{lemma-flowcol} can be used to decide in polynomial time
whether a given precoloring $\psi$ of the outer face extends to a $3$-coloring of $G$: try all possible
$\psi$-balanced layouts $q$ for $G$ (whose number is bounded by a function of $c$) and for each of them, decide whether
the terminals of $G^{q,\psi}$ are separated by an edge-cut of size less than $c(q,\psi)$ using a maximum flow algorithm.

In order to apply Lemma~\ref{lemma-flowcol} efficiently, we describe the structure of small edge-cuts separating the terminals
in $G^{q,\psi}$.
\begin{lemma}\label{lemma-cut}
Let $G$ be a connected plane triangle-free graph with the outer face $C$ bounded by a cycle and let $\psi$ be
a $3$-coloring of $C$ that does not extend to a $3$-coloring of $G$.  If $q$ is a $\psi$-balanced layout in $G$,
then there exists a subgraph $K_0\subseteq G$ such that either
\begin{itemize}
\item[\textrm{(a)}] $K_0$ is a path with both ends in $C$ and no internal vertex in $C$, and if $P$ is a path in $C$ joining the endvertices
of $K_0$, $n_s$ is the number of source edges of $P$, $n_t$ is the number of sink edges of $P$ and $m$ is the sum of the values
of $q$ over all faces of $G$ drawn in the open disk bounded by the cycle $P+K_0$, then
$|n_s+m-n_t|>|K_0|$.  In particular, $|P|+|m|>|K_0|$.  Or, 
\item[\textrm{(b)}] $K_0$ is a cycle with at most one vertex in $C$, and if $m$ is the sum of the values
of $q$ over all faces of $G$ drawn in the open disk bounded by $K_0$, then $|m|>|K_0|$.
\end{itemize}
See Figure~\ref{fig-K0} for an illustration.
\end{lemma}
\begin{figure}
\center{\includegraphics[scale=1]{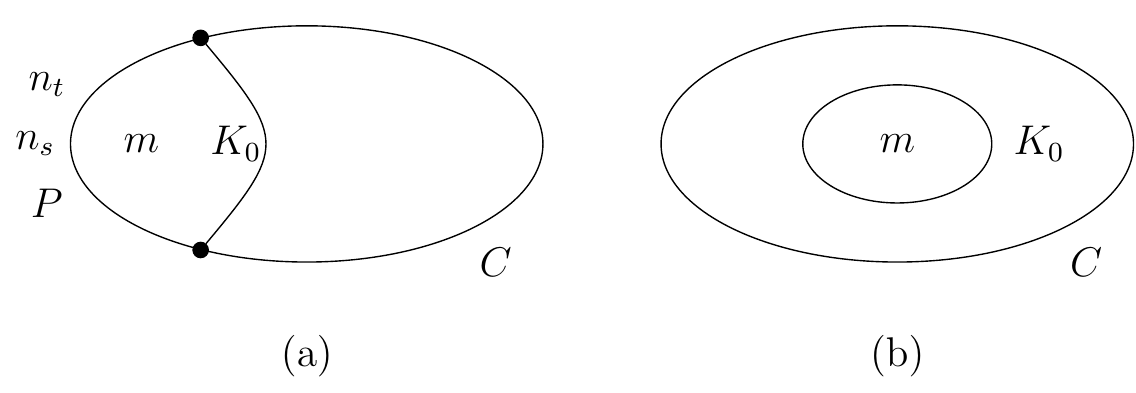}}
\caption{Possibilites in Lemma~\ref{lemma-cut}.}\label{fig-K0}
\end{figure}
\begin{proof}
By Lemma~\ref{lemma-flowcol}, there exists an edge-cut in $G^{q,\psi}$ smaller than $c(q,\psi)$ and separating the terminals.
Let us choose such an edge-cut $K$ so that $K$ has as few edges not incident with the terminals as possible, and subject to that
$|K|$ is minimal.
Let $s$ and $t$ be the terminals of $G^{q,\psi}$ and let $S$ and $T$ be the sets of edges of $G^{q,\psi}$ incident with
$s$ and $t$, respectively.  Let $K_0=K\setminus (S\cup T)$; note that since $|K|<c(q,\psi)$ and $G^{q,\psi}-\{s,t\}$ is connected,
the set $K_0$ is not empty.  Therefore, the minimality of $K_0$ implies that $G^{q,\psi}-\{s,t\}-K$ has at least two components.

Consider a connected component $A$ of $G^{q,\psi}-\{s,t\}-K$.  Let $Y$ be the set of edges of $G^{q,\psi}$ connecting a vertex of $A$
with a vertex of $V(G^{q,\psi})\setminus A$, let $Y_0=Y\cap K_0$, $Y_s=Y\cap S$ and $Y_t=Y\cap T$.
As $G^{q,\psi}-\{s,t\}$ is connected, the set $Y_0$ is not empty.  By the minimality of $K$, either
$Y_s$ is non-empty, $Y_s\cap K=\emptyset$ and $Y_t\subseteq K$, or $Y_t$ is non-empty, $Y_t\cap K=\emptyset$ and $Y_s\subseteq K$.
In the former case we say that $A$ is an $s$-component, in the latter case a $t$-component.

By symmetry, we can assume that $A$ is a $t$-component.  Suppose that $|Y_t|\le |Y_s|+|Y_0|$.  By the minimality of $K_0$,
each edge of $Y_0$ is incident with an $s$-component, and thus $K'=(K\setminus (Y_s\cup Y_0))\cup Y_t$ is an edge-cut
separating $s$ from $t$ of size at most $|K|$.  Furthermore, since $Y_0$ is nonempty, $K'$ would contradict
the assumption that $K$ was chosen with $K_0$ as small as possible.

It follows that
\claim{cl-comp}{each $t$-component satisfies $|Y_t|>|Y_s|+|Y_0|$, and each $s$-component satisfies $|Y_s|>|Y_t|+|Y_0|$.}
Note that $K''=Y_s\cup Y_0\cup (T\setminus Y_t)$ is an edge-cut separating $s$ from $t$ of size $|Y_s|+|Y_0|+|T|-|Y_t|<|T|=c(q,\psi)$.
By the minimality of $K_0$, we have $Y_0=K_0$.  As this observation applies to every connected component of $G^{q,\psi}-\{s,t\}-K$,
we conclude that $G^{q,\psi}-\{s,t\}-K$ has exactly two connected components (an $s$-component and a $t$-component).

We can interpret $K_0$ as a subgraph of $G$ by duality.
Thus, the conclusion of the previous paragraph is equivalent to the claim that the subgraph of $G$ consisting of $C$ and $K_0$ has exactly two internal faces.
It follows that the edges of $K_0$ form either a path joining two distinct vertices of $C$, or a cycle intersecting $C$ in at
most one vertex.  

In the former case, let $P$ be a path in $C$ joining the endvertices of $C$, and let $n_s$, $n_t$ and $m$ be defined as in the
statement (a) of the lemma.  Let $A$ be the component of $G^{q,\psi}-\{s,t\}-K$ whose vertices correspond to the faces of $G$ drawn
in the open disk bounded by $P+K_0$.  Note that $|Y_s|-|Y_t|=n_s-n_t+m$.  By \refclaim{cl-comp}, we have
$|n_s-n_t+m|>|K_0|$.

In the latter case, $G^{q,\psi}-\{s,t\}-K$ has a component $A$ not incident with any edges of $C$; its vertices correspond to the
faces of $G$ drawn in the open disk bounded by $K_0$.  Let $Y_0$, $Y_s$ and $Y_t$ be defined as before, and note that $m=|Y_s|-|Y_t|$.
By \refclaim{cl-comp}, we have $|m|=||Y_s|-|Y_t||>|Y_0|=|K_0|$.
\end{proof}

\section{Faces in critical graphs}\label{sec:faces}

Note that Lemma~\ref{lemma-flowcol} is most useful for graphs with almost all faces of length $4$, as
then there are only a few choices for the layout.  Recall that for a plane graph $G$,  we use $S(G)$ to denote the
multiset of lengths of the internal $(\ge\!5)$-faces of $G$.  Let $\GG_{g,k}$ denote the set of
plane graphs of girth at least $g$ and with outer face formed by a cycle $C$ of length $k$
that are $C$-critical.  Let $\Ss_{g,k}=\{S(G):G\in\GG_{g,k}\}$.  

Note that $\Ss_{5,k}$ is finite for every $k$ by Thomassen~\cite{thomassen-surf}.  However, we only need to know
the sets $\Ss_{5,k}$ for $k\le 8$.
\begin{theorem}[Thomassen~\cite{thom-torus}]\label{thm-thom}
Let $G$ be a plane graph of girth at least five with outer face bounded by a cycle $C$ of length at most $8$.
The graph $G$ is $C$-critical if and only if $G$ is an 8-cycle with a chord.
\end{theorem}
This implies that $\Ss_{5,k}=\emptyset$ for $k\le 7$ and that $\Ss_{5,8}=\{\{5,5\}\}$.
Dvo\v{r}\'ak et al.~\cite{trfree4}
proved that $\Ss_{4,k}$ is finite for every $k$ (actually, the sum of each multiset in $\Ss_{4,k}$
is bounded by a linear function of $k$).  Furthermore, they proved that $\{k-2\}$ belongs to $\Ss_{4,k}$
and every other element of $\Ss_{4,k}$ has maximum at most $k-3$.

Let $S_1$ and $S_2$ be multisets of integers.  We say that $S_2$ is a \emph{one-step refinement} of $S_1$
if there exist $k\in S_1$ and a set $Z\in \Ss_{4,k}\cup \Ss_{4,k+2}$ such that $S_2=(S_1\setminus \{k\})\cup Z$.
We say that $S_2$ is a \emph{refinement} of $S_1$ if it can be obtained from $S_2$ by a (possibly empty)
sequence of one-step refinements.

\begin{lemma}[Dvo\v{r}\'ak et al.~\cite{trfree4}]\label{lemma-faces}
For every $k\ge 7$, each element of $\Ss_{4,k}$ other than $\{k-2\}$ is a refinement of an element of $\Ss_{4,k-2}\cup \Ss_{5,k}$.
\end{lemma}

In particular, together with Theorems~\ref{thm-gimbel} and \ref{thm-thom} this implies that $\Ss_{4,k}=\emptyset$ for $k\le 5$,
$\Ss_{4,6}=\{\emptyset\}$, $\Ss_{4,7}\subseteq\{\{5\}\}$ and $\Ss_{4,8}\subseteq\{\emptyset, \{5,5\}, \{6\}\}$.

\section{Quadrangulations and graphs with a $(k-2)$-face}\label{sec:k-2}

As a first application of Lemma~\ref{lemma-flowcol}, let us consider quadrangulations.

\begin{theorem}\label{thm:only4}
Let $G$ be a connected triangle-free plane graph with outer face bounded by a cycle $C$ of length $k\ge 6$.
Suppose that all internal faces of $G$ have length $4$.  The graph $G$ is $C$-critical if and only if
$G$ contains no separating $4$-cycles.
\end{theorem}
\begin{proof}
If $G$ is $C$-critical, then it does not contain separating $4$-cycles by Lemma~\ref{lemma-critin}
and Theorem~\ref{thm-gimbel}.

Suppose now conversely that $G$ does not contain separating $4$-cycles.  Let $G_0$ be a proper subgraph of $G$ containing $C$.
Since $G$ does not contain separating $4$-cycles, it follows that $G_0$ has a face $f$ of length at least $6$. 
Let $G_1\supseteq G_0$ be the subgraph of $G$ consisting of all vertices and edges not drawn in the interior of $f$.
As $G$ is bipartite, it does not contain separating $(\le\!5)$-cycles, and thus the dual $G^\star_1$ of $G_1$ does
not contain an edge-cut of size at most $5$ separating $C$ from $f$.  Therefore, there exist $6$ pairwise edge-disjoint
paths $P_1$, \ldots, $P_6$ from $f$ to $C$ in $G^\star_1$.  Note that all vertices of $G^\star_1-E(P_1\cup\ldots\cup P_6)$
have even degree, and thus $G^\star_1-E(P_1\cup\ldots\cup P_6)$ is a union of pairwise edge-disjoint cycles.
We orient the edges of each of the cycles in one (arbitrary) direction and direct the edges of $P_1$, \ldots, $P_6$ towards $C$.
This orientation of $G^\star_1$ corresponds to a nowhere-zero $\mathbb Z_3$-flow, giving a $3$-coloring $\varphi$ of $G_1$.
Let $\psi$ be the restriction of $\varphi$ to $C$, and observe that $C$ contains $k/2+3$ sink edges and $k/2-3$ source
edges with respect to this coloring.

Note that $G$ has only one layout (assigning $0$ to every internal face), and this layout is not $\psi$-balanced.
By Lemma~\ref{lemma-flowcol}, $\psi$ does not extend to a $3$-coloring of $G$.  On the other hand, $\psi$ extends
to a $3$-coloring of $G_0$, since $G_0$ is a subgraph of $G_1$.  We conclude that for every proper subgraph of $G$ containing $C$,
there exists a precoloring of $C$ that extends to this subgraph but not to $G$, and thus $G$ is $C$-critical.
\end{proof}

Next, we deal with the graphs with an internal face of length $k-2$.
Let $r(k)=0$ if $k\equiv 0\pmod{3}$, $r(k)=2$ if $k\equiv 1\pmod{3}$ and $r(k)=1$ if $k\equiv 2\pmod{3}$.

\begin{theorem}\label{thm:n-2}
Let $G$ be a connected triangle-free plane graph with outer face bounded by a cycle $C$ of length $k\ge 7$.
Suppose that $f$ is an internal face of $G$ of length $k-2$ and that all other internal faces
of $G$ have length $4$.  The graph $G$ is $C$-critical if and only if
\begin{itemize}
\item[(a)] $f\cap C$ is a path of length at least $r(k)$ (possibly empty if $r(k)=0$),
\item[(b)] $G$ contains no separating $4$-cycles, and
\item[(c)] for every $(\le\!k-1)$-cycle $K\neq f$ in $G$, the interior of $K$ does not contain $f$.
\end{itemize}
Furthermore, in a graph satisfying these conditions, a precoloring $\psi$ of $C$ extends to a $3$-coloring of
$G$ if and only if $E(C)\setminus E(f)$ contains both a source edge and a sink edge with respect to $\psi$.
\end{theorem}
\begin{proof}
If $G$ is $C$-critical, then it does not contain separating $4$-cycles by Lemma~\ref{lemma-critin}
and Theorem~\ref{thm-gimbel}.  Furthermore, by Lemma~\ref{lemma-faces}, if $m\le k-1$, then no element of $\Ss_{4,m}$
contains $k-2$, and thus Lemma~\ref{lemma-critin} implies that if $G$ is $C$-critical, then it satisfies (c).
Assume from now on that $G$ satisfies (b) and (c).

Suppose that $\psi$ is a $3$-coloring of $C$ that does not extend to a $3$-coloring of $G$.
Let $d$ be the difference between the number of sink and source edges of $C$ with respect to $\psi$.
If $|d|>k-2$, then since $k$ and $d$ have the same parity, it follows that $|d|=k$.  In this case $k\equiv 0\pmod{3}$
and $r(k)=0$, and thus (a) is trivially true.  Furthermore, either all edges of $C$ are source or all of them are sink.

If $d\le k-2$, then let $q$ be the layout for $G$ such that $q(f)=d$.  Note that $q$ is $\psi$-balanced.
Let $K_0$ be the subgraph of $G$ obtained by Lemma~\ref{lemma-cut}.
As $G$ satisfies (c), $K_0$ is not a cycle.  Therefore, $K_0$ is a path joining two distinct vertices of $C$.
Let $P$ be the subpath of $C$ joining the endvertices of $K_0$ such that the open disk bounded by
$P+K_0$ does not contain $f$, let $n_s$ be the number of source edges in $P$ and $n_t$ the number of sink edges in $P$.
By Lemma~\ref{lemma-cut}, we have $|P|>|K_0|$.  Let $R$ be the path $C-P$
and consider the cycle $Z=R+K_0$.  Note that $|Z|=k+|K_0|-|P|<k$.  By (c), we conclude that $Z$ is the boundary of $f$.
As $|f|=k-2$, it follows that $|K_0|=|P|-2$, and since $n_s+n_t=|P|$ and $|n_s-n_t|>|K_0|$, it follows that $n_s=0$ or $n_t=0$,
i.e., either all edges of $E(C)\setminus E(f)$ are source or all of them are sink.

Let $m_s$ be the number of source edges of $C$ and $m_t$ the number of sink edges of $C$. Note that $m_t-m_s=d\equiv 0\pmod3$,
and thus $m_s+m_t\equiv 2m_s\equiv 2m_t\pmod3$.  It follows that $m_s\equiv m_t\equiv r(k)\pmod3$.  Since $R=f\cap C$ contains
either all source edges or all sink edges of $C$, it follows that $f\cap C$ is a path of length at least $r(k)$.

Therefore, if $G$ is $C$-critical, then it satisfies (a).  Furthermore, if $G$ is a graph satisfying (b) and (c),
then every precoloring of $C$ such that $E(C)\setminus E(f)$ contains both a source edge and a sink edge
extends to a $3$-coloring of $G$.

Suppose now that $G$ satisfies (a), (b) and (c).  Let $\psi$ be a $3$-coloring of $C$ such that all edges of
$E(C)\setminus E(f)$ are source (such a coloring exists by (a)).  Let $q$ be the unique $\psi$-balanced layout for $G$
and let $K$ consist of $E(f)\setminus E(C)$ and of all edges of $G^{q,\psi}$ joining $f$ with $s$.
Note that $|E(C)\setminus E(f)|=|E(f)\setminus E(C)|+2$, and thus $K$ is an edge-cut separating the terminals in $G^{q,\psi}$
of size $c(q,\psi)-2$.  By Lemma~\ref{lemma-flowcol}, $\psi$ does not extend to a $3$-coloring of $G$.

It follows that $G$ contains a $C$-critical subgraph $G_0$.  By (c), $G_0$ has an internal face of length at least $k-2$.
By Lemma~\ref{lemma-faces}, the only element of $\Ss_{4,k}$ whose maximum is at least $k-2$ is $\{k-2\}$,
and thus all other internal faces of $G_0$ have length four.  By (b), it follows that $G=G_0$, and thus $G$ is $C$-critical.
\end{proof}

\begin{figure}
\center{\includegraphics{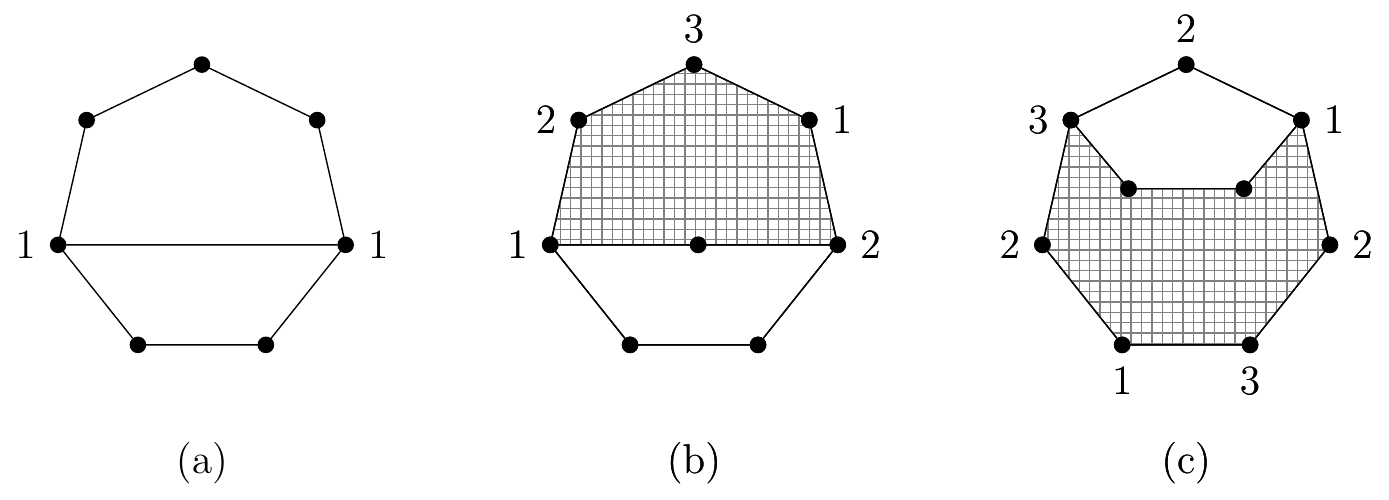}}
\caption{Critical graphs with a precolored $7$-face.}\label{fig-7cyc}
\end{figure}

Since $\Ss_{4,7}\subseteq \{\{5\}\}$, Theorem~\ref{thm:n-2} fully characterizes plane triangle-free graphs critical with
respect to a precolored $7$-face.  The following corollary additionally discusses the three possible cases of the intersection of the $5$-face with
the $7$-face in such a graph.

\begin{corollary}\label{corollary-7cyc}
Let $G$ be a plane triangle-free graph with outer face bounded by a cycle $C=c_1\ldots c_7$ of length $7$.  The graph $G$ is $C$-critical
and $\psi$ is a $3$-coloring of $C$ that does not extend to a $3$-coloring of $G$ if and only if $G$ contains no separating cycles of length at most five
and one of the following propositions is satisfied up to relabelling of vertices (see Figure~\ref{fig-7cyc} for an illustration).
\begin{itemize}
\item[\textrm(a)] The graph $G$ consists of $C$ and the edge $c_1c_5$, and $\psi(c_1)=\psi(c_5)$.
\item[\textrm(b)] The graph $G$ contains a vertex $v$ adjacent to $c_1$ and $c_4$, the cycle $c_1c_2c_3c_4v$ bounds a 5-face and every face drawn inside
the $6$-cycle $vc_4c_5c_6c_7c_1$ has length four; furthermore, $\psi(c_4)=\psi(c_7)$ and $\psi(c_5)=\psi(c_1)$.
\item[\textrm(c)] The graph $G$ contains a path $c_1uvc_3$ with $u,v\not\in V(C)$, the cycle $c_1c_2c_3vu$ bounds a 5-face and every face drawn inside
the $8$-cycle $uvc_3c_4c_5c_6c_7c_1$ has length four; furthermore, $\psi(c_3)=\psi(c_6)$, $\psi(c_2)=\psi(c_4)=\psi(c_7)$
and $\psi(c_1)=\psi(c_5)$.
\end{itemize}
\end{corollary}

\section{Graphs with precolored $8$-face}\label{sec:8face}

Finally, we consider the plane triangle-free graphs critical with respect to a precolored face of length $8$.

\begin{proof}[Proof of Theorem~\ref{thm:8cycleprecise}]
Suppose that $G$ is $C$-critical.  By Lemma~\ref{lemma-critin} and Theorem~\ref{thm-gimbel}, $G$ does not
contain any separating cycles of length at most five and the interior of every non-facial 6-cycle
contains only faces of length four.  By Lemma~\ref{lemma-faces}, we have $S(G)=\emptyset$ or
$S(G)=\{6\}$ or $S(G)=\{5,5\}$.  If $S(G)=\emptyset$, then $G$ satisfies (a).  If $S(G)=\{6\}$, then
$G$ satisfies (b) by Theorem~\ref{thm:n-2}.  Therefore, suppose that $S(G)=\{5,5\}$, and let $f_1$ and $f_2$ be
the $5$-faces of $G$.

Let $\psi$ be a $3$-coloring of $C$ by colors $\{1,2,3\}$ that does not extend to a $3$-coloring of $G$.
By symmetry, we can assume that $C$ contains at least as many source edges as sink ones.  It follows
that $C$ contains either $4$ or $7$ source edges.
If $C$ has $7$ source edges, then let $q$ be the layout for $G$ such that $q(f_1)=q(f_2)=-3$.
Note that $q$ is $\psi$-balanced, and consider the subgraph $K_0$ of $G$ obtained by Lemma~\ref{lemma-cut}.
As $|q(f_1)+q(f_2)|=6$ and neither $f_1$ nor $f_2$ are contained inside a separating $(\le\!6)$-cycle,
it follows that $K_0$ is a path joining two distinct vertices of $C$.  Let $P$ be a subpath of $C$ joining
the endpoints of $K_0$ such that the open disk $\Delta$ bounded by $P+K_0$ contains at most one of $f_1$ and $f_2$.
Let $n_s$ be the number of source edges in $P$ and $n_t$ the number of sink edges in $P$.
If $\Delta$ does not contain any $5$-face, then $|P|>|K_0|$ by 
Lemma~\ref{lemma-cut}.  Since all faces in $\Delta$ have length $4$, it follows that $|P|+|K_0|$ is even,
and thus $|P|\ge |K_0|+2$.  If $\Delta$ contains a $5$-face (say $f_1$), then since $C$ has $7$ source edges, we can by
symmetry between $f_1$ and $f_2$ assume that $n_s\ge 4$, and thus $|P|-3\ge n_s-n_t-3=|n_s-n_t+q(f_1)|>|K_0|$ by Lemma~\ref{lemma-cut}.
In both cases, we have $|P|\ge |K_0|+2$.  Let $R=C-P$ and note that $R+K_0$ is a cycle of length
at most $|C|-|P|+|K_0|\le |C|-2=6$.  However, by the choice of $P$, the open disk bounded by $R+K_0$ contains
a $5$-face, contrary to the assumptions of Theorem~\ref{thm:8cycleprecise}.

Therefore, $C$ has $4$ source edges, and $G$ has two $\psi$-balanced layouts $q_1$ and $q_2$
such that $q_i(f_i)=3$ and $q_i(f_{3-i})=-3$ for $i\in \{1,2\}$.
Let $K_1$ and $K_2$ be the subgraphs of $G$ obtained by Lemma~\ref{lemma-cut} applied to $q_1$ and $q_2$,
respectively.  As $|q_i(f_j)|=3$ and $q_i(f_1)+q_i(f_2)=0$ for $i,j\in \{1,2\}$, the case (b) of Lemma~\ref{lemma-cut}
cannot apply, and thus both $K_1$ and $K_2$ are paths.  Let $v_1$ and $w_1$ be the endpoints of $K_1$ and let $v_2$ and $w_2$ be the endpoints of $K_2$.

Suppose that there exists a path $P\subset C$ joining $v_i$ with $w_i$ for some $i\in\{1,2\}$
such that all faces drawn in the open disk bounded by $P+K_i$ have length $4$.  By Lemma~\ref{lemma-cut},
we have $|P|>|K_i|$, and since $P+K_i$ has even length, $|P|\ge |K_i|+2$.  We conclude that the cycle
$(C-P)+K_i$ has length at most $|C|-|P|+|K_i|\le 6$, and since the open disk bounded by $(C-P)+K_i$ contains
two $5$-faces, it contradicts the assumptions of Theorem~\ref{thm:8cycleprecise}.

Consequently, there is no such path.  For $i\in\{1,2\}$, let $P_{i}\subset C$ be the path
joining $v_i$ with $w_i$ such that the open disk $\Delta_{i}$ bounded by $P_{i}+K_i$ contains
$f_i$.  Let $n_i^s$ and $n_i^t$ denote the number of source and sink edges, respectively, of $P_{i}$, for $i\in\{1,2\}$.
By Lemma~\ref{lemma-cut}, we have $|n^s_i-n^t_i+3|>|K_i|$.  Since $n^t_i\le 4$ and $|K_i|\ge 1$, we have
$|n^s_i-n^t_i+3|=n^s_i-n^t_i+3$, and thus $n^s_i-n^t_i+2\ge |K_i|$.  As $\Delta_{i}$ contains one $5$ face
and all other faces in $\Delta_{i}$ have length $4$, the cycle $K_i+P_{i}$ has odd length, and thus
$|P_{i}|$ and $|K_i|$ have opposite parity.  Since $|P_i|$ and $n^s_i-n^t_i$ have the same
parity, we can improve the inequality to
\claim{cl-mainin}{$n^s_i-n^t_i+1\ge |K_i|$.}

\begin{figure}
\center{\includegraphics{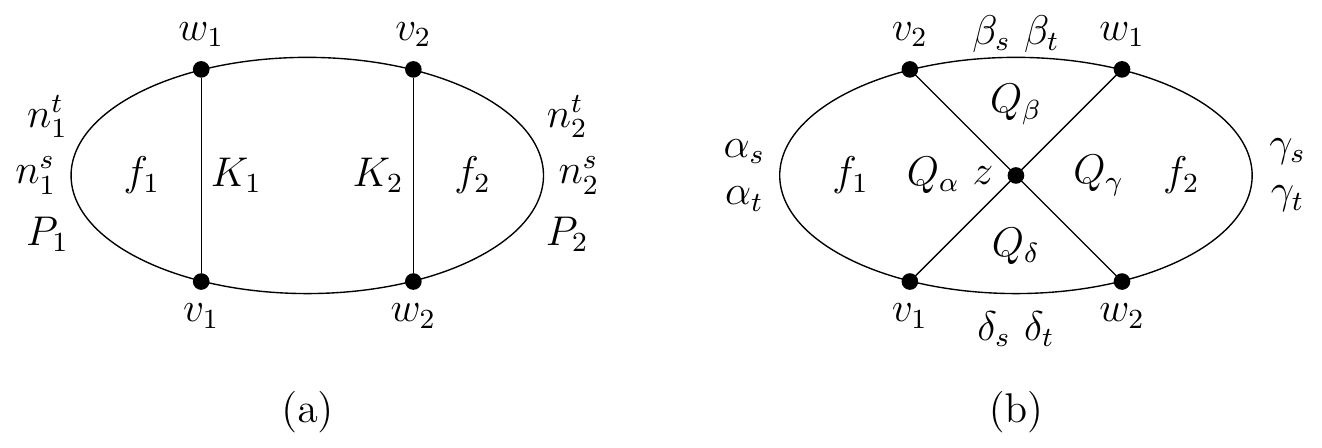}}
\caption{Possible cuts in $G$.}\label{fig-ff}
\end{figure}

Suppose first that we can choose the labels so that the order of the endpoints along $C$ is $v_1$, $w_1$, $v_2$, $w_2$
(this is always the case if the endpoints of $K_1$ and $K_2$ are not pairwise distinct).  Let $\psi'$ be the $3$-coloring
of $C$ defined by $\psi'(x)=4-\psi(x)$.  The source edges with respect to $\psi$ are sink edges with respect to $\psi'$ and vice
versa.  Consequently, replacing $\psi$ with $\psi'$ swaps the roles of the paths $K_1$ and $K_2$.  Therefore,
by making this replacement if necessary, we can assume that $P_{1}$ and $P_{2}$ are edge-disjoint,
as in Figure~\ref{fig-ff}(a).
For $i\in\{1,2\}$, the cycle $K_i+P_{i}$ has length at least five, i.e.,
\claim{cl-lenin}{$n^s_i+n^t_i+|K_i|\ge 5$.}
Summing this inequality with \refclaim{cl-mainin}, we obtain $n^s_i\ge 2$.  However, since $P_{1}$ and $P_{2}$
are edge-disjoint and $C$ has exactly $4$ source edges, we have $n^s_1+n^s_2\le 4$.  We conclude
that $n^s_i=2$ for $i\in\{1,2\}$ and that equality holds in \refclaim{cl-mainin} and \refclaim{cl-lenin}.
In particular, $K_i+P_{i}$ is a $5$-cycle, and since $G$ does not contain separating $5$-cycles,
it follows that $K_i+P_{i}=f_i$.  Therefore, $f_i\cap C=P_{i}$ is a path of length at least $n^s_i=2$,
and $G$ satisfies (c).

Finally, consider the case that the endpoints of $K_1$ and $K_2$ are pairwise distinct and their order along $C$ is $v_1$, $v_2$, $w_1$, $w_2$.
Let $\alpha$, $\beta$, $\gamma$ and $\delta$ be the subpaths of $C$ between $v_1$ and $v_2$, between $v_2$ and $w_1$, between $w_1$ and $w_2$ and between $w_2$ and $v_1$,
respectively, chosen so that the paths $\alpha$, $\beta$, $\gamma$ and $\delta$ are pairwise edge-disjoint.
By planarity, the paths $K_1$ and $K_2$ intersect.
Let $Q_\alpha$ denote the walk between $v_1$ and $v_2$ such that the concatenation of $\alpha$ with $Q_\alpha$ is the boundary walk
of the internal face of the graph $C+K_1+K_2$ incident with $\alpha$.  Let $\alpha_s$ and $\alpha_t$ denote the number of source and sink edges
of $\alpha$, respectively.  Define $Q_x$, $x_s$ and $x_t$ for $x\in\{\beta,\gamma,\delta\}$ analogously.  By symmetry, we can assume
that $f_1$ is contained in the open disk bounded by $\alpha+Q_\alpha$ and $f_2$ is contained in the closed disk bounded by $\gamma+Q_\gamma$,
as in Figure~\ref{fig-ff}(b).
Note that $P_{1}=\alpha+\beta$ and $P_{2}=\beta+\gamma$, and thus by \refclaim{cl-mainin}, we have
$\alpha_s+\beta_s-\alpha_t-\beta_t+1\ge |K_1|$ and $\beta_s+\gamma_s-\beta_t-\gamma_t+1\ge |K_2|$.
Furthermore, $\alpha+Q_\alpha$ and $\gamma+Q_\gamma$ have length at least $5$,
and $|Q_\alpha|+|Q_\gamma|\le |K_1|+|K_2|$, implying that $\alpha_s+\alpha_t+\gamma_s+\gamma_t+|K_1|+|K_2|\ge 10$.
Summing these inequalities, we obtain $\alpha_s+\beta_s+\gamma_s-\beta_t\ge 4$.  As $\alpha_s+\beta_s+\gamma_s+\delta_s=4$,
this implies that $-\beta_t\ge \delta_s$, and as $\beta_t$ and $\delta_s$ are nonnegative, we have $\beta_t=\delta_s=0$.
Furthermore, all the inequalities must hold with equality, and in particular $\alpha+Q_\alpha$ and $\gamma+Q_\gamma$ have length $5$.
As $G$ does not contain separating $5$-cycles, we have $f_1=\alpha+Q_\alpha$ and $f_2=\gamma+Q_\gamma$.  Also,
$|Q_\alpha|+|Q_\gamma|=|K_1|+|K_2|$, and thus every edge of $K_1\cup K_2$ is incident with $f_1$ or $f_2$.
We conclude that the boundaries of $f_1$ and $f_2$ intersect; let $z$ be an arbitrary common vertex of $f_1$ and $f_2$.

Since $\beta_t=0=\delta_s=0$, all edges of $\beta$ are source and all edges of $\delta$ are sink.  Let $c_1$ be the vertex of $\alpha$
at distance $\alpha_s$ from $v_2$, and $c_5$ the vertex of $\gamma$ at distance $\gamma_s$ from $w_1$.  
Note that the distance between $c_1$ and $c_5$ in $C$ is $\alpha_s+\beta_s+\gamma_s=4-\delta_s=4$, and thus we can label the vertices of $C$ as $c_1c_2\ldots c_8$ in order.
Let $W_1$ be the path between $c_1$ and $z$ in $\alpha\cup K_1$, let $Z_1$ be the path between $c_1$ and $z$ in $\alpha\cup K_2$,
let $W_2$ be the path between $c_5$ and $z$ in $\gamma\cup K_2$ and let $Z_2$ be the path between $c_5$ and $z$ in $\gamma\cup K_1$.
Note that $W_1+Z_1=f_1$ and $W_2+Z_2=f_2$, and thus $|W_1|+|Z_1|=5$ and $|W_2|+|Z_2|=5$.
As $\alpha_s+\beta_s-\alpha_t-\beta_t+1=|K_1|$ and $\alpha_s+\beta_s+\gamma_s=4$, we have $\alpha_t+|K_1|+\gamma_s=5-\beta_t=5$, i.e.,
$|W_1|+|Z_2|=5$.  Symmetrically, $|Z_1|+|W_2|=5$.  It follows that $|W_1|=|W_2|$ and $|Z_1|=|Z_2|$.
As the disk bounded by the closed walk $W_1+W_2+c_1c_2c_3c_4c_5$ contains $5$-faces $f_1$ and $f_2$,
we have $|W_1|+|W_2|+4>6$.  Consequently, $|W_1|>1$, and symmetrically $|Z_1|>1$.  By symmetry, we can assume that
$|W_1|=|W_2|=2$ and $|Z_1|=|Z_2|=3$, and we conclude that $G$ satisfies (d).

Suppose now that $G$ is a graph satisfying (a), (b), (c) or (d).  If $G$ satisfies (a), then it is $C$-critical by
Theorem~\ref{thm:only4}.  If $G$ satisfies (b), then it is $C$-critical by Theorem~\ref{thm:n-2}.
Let us consider the case that $G$ satisfies (c) or (d).  If $G$ satisfies (c), then let $\psi$ be a $3$-coloring of $C$
with exactly $4$ source edges, two of them incident with $f_1$ and two of them incident with $f_2$.
If $G$ satisfies (d), then let $\psi$ be a $3$-coloring of $C$ with exactly $4$ source edges $c_1c_2$,
$c_2c_3$, $c_3c_4$ and $c_4c_5$.  Note that $q_1$ and $q_2$ are the only $\psi$-balanced layouts for $G$.
Observe that both $G^{q_1,\psi}$ and $G^{q_2,\psi}$ contain an edge-cut of size at most $5$ separating the
terminals, while $c(q_1,\psi)=c(q_2,\psi)=7$.  By Lemma~\ref{lemma-flowcol}, $\psi$ does not extend to a $3$-coloring of $G$,
and thus $G$ has a $C$-critical subgraph $G_0$.  By Lemma~\ref{lemma-faces}, we have $S(G_0)\in\{\emptyset, \{6\}, \{5,5\}\}$.
As $G$ does not contain separating $(\le\!5)$-cycles and the interior of every 6-cycle of $G$ contains only $4$-faces,
we conclude that $G=G_0$, and thus $G$ is $C$-critical.
\end{proof}

\paragraph{Acknowledgement.}
We thank Alexandr Kostochka for encouraging us to write this paper.

\bibliographystyle{acm}
\bibliography{8cyc}
\end{document}